\documentclass{amsart}
\usepackage[centertags]{amsmath}
\usepackage{amstext,amssymb,amsopn,amsthm}

\title
{The best constant in a fractional Hardy inequality}
\author[K{.} Bogdan and B{.} Dyda]
{Krzysztof Bogdan and Bart{\l}omiej Dyda}
\date{April 4, 2008}
\address{Institute of Mathematics and Computer Science, Wroc{\l}aw University of Technology,
Wybrze\.ze Wyspia\'nskiego 27,
50-370 Wroc{\l}aw, Poland}
\email{bogdan@pwr.wroc.pl bdyda@pwr.wroc.pl}

\theoremstyle{plain}
\newtheorem{thm}{Theorem}

\newtheorem{lem}[thm]{Lemma}

\theoremstyle{definition}

\theoremstyle{remark}
\newtheorem*{rem*}{Remark}

\newtheorem*{ack}{Acknowledgement}

\newcommand{\Rd}{{\mathbb{R}^d}}

\newcommand{\R}{\mathbb{R}}

\DeclareMathOperator{\dist}{dist}

\DeclareMathOperator{\supp}{supp}

\date{\today}

\begin{document}
\sloppy \footnotetext{\emph{2000 Mathematics Subject
Classification:} Primary 26D10, secondary 31C25, 31B25, 46E35.\\
\emph{Key words and phrases:} fractional Hardy inequality,
  best constant, half-space, fractional Laplacian, censored stable process.
\\ Research partially supported by KBN 1 P03A 026 29}

\begin{abstract}
We prove an optimal Hardy inequality
for the fractional Laplacian on the half-space.
\end{abstract}
\maketitle
\section{Main result and discussion}\label{sec:i}
Let $0<\alpha<2$ and $d=1,2,\ldots$.
The purpose of this note is to prove the following Hardy-type inequality in
the half-space $D=\{x=(x_1,\ldots,x_d)\in \Rd:\,x_d>0\}$.
\begin{thm}\label{hardy}
For every $u\in C_c(D)$,
\begin{equation}\label{Hardyin}
\frac{1}{2}
\int_D \! \int_D
\frac{(u(x)-u(y))^2}{|x-y|^{d+\alpha}} \,dx\,dy
\geq {\kappa_{d,\alpha}}
\int_D {u^2(x)}
\;
x_d^{-\alpha}dx\,,
\end{equation}
where
\begin{equation}
  \label{eq:ns}
\kappa_{d,\alpha}=\frac{\pi^{\frac{d-1}{2}}
  \Gamma(\frac{1+\alpha}{2})}{\Gamma(\frac{\alpha+d}{2})}
\frac{B\left(\frac{1+\alpha}{2}, \frac{2-\alpha}{2}\right)
  -2^{\alpha}}{\alpha2^{\alpha}}\,,
\end{equation}
and {\rm (\ref{Hardyin})} fails to hold for some $u\in C_c(D)$ if $\kappa_{d,\alpha}$ is replaced by a
bigger constant.
\end{thm}
Here $B$ is the Euler beta function, and 
$C_c(D)$ denotes the class of all the continuous functions
$u\,:\; \Rd\to \R$ with compact support in $D$.
On the right-hand side of (\ref{Hardyin}) we note the infinite measure
$x^{-\alpha}_ddx$, where
$x_d$ equals the distance of $x=(x_1,\ldots,x_d)\in D$ to the complement
of $D$.
Analogous Hardy inequalities, 
involving the distance to the complement of rather general domains,
and arbitrary positive exponents of integrability of functions $u$, were proved with
{\it rough} constants in \cite{Dijm2004} (see also \cite{MR1676324, ChenSong, Dyda2}).
Thus the focus in Theorem~\ref{hardy} is on optimality
of $\kappa_{d,\alpha}$.
We note that $\kappa_{d,1}=0$ and $\kappa_{d,\alpha}>0$ if $\alpha\neq
1$ (see the proof of Lemma~\ref{propapp}).

Theorem~\ref{hardy} may be viewed as an application of ideas of Ancona  \cite{Ancona}
and  Fitzsimmons \cite{Fitz}. Indeed, consider
the Dirichlet form $\mathcal{E}$, with domain $Dom(\mathcal{E})$, and
the generator $\mathcal{L}$, with domain $Dom(\mathcal{L})$,
of a symmetric Markov process 
(\cite{MR1303354}, \cite{MR1214375}, \cite{MR1873235}), and a function $w>0$, and a measure
$\nu\geq0$ on the state space. The following result was proved by
Fitzsimmons in \cite{Fitz}.
\begin{equation}
  \label{eq:tf}
  \mbox{If } \mathcal{L}w\leq -w\nu\quad \mbox{then}\quad \mathcal{E}(u,u)\geq
  \int u^2 \,d\nu\,,\quad u\in Dom(\mathcal{E})\,. 
\end{equation}
Thus, every superharmonic function $w$
(i.e. $w\geq 0$ such that $\mathcal{L}w\leq 0$) yields a Hardy-type
inequality with integral weight $\nu=-\mathcal{L}w/w$. For instance, in the proof of
Theorem~\ref{hardy} we will use $w(x)=x_d^{(\alpha-1)/2}$. 
Full details of (\ref{eq:tf}), and a converse result are given in \cite[Theorem 1.9]{Fitz}. 
Recall that
$$\mathcal{E}(u,v)=-(Lu,v)\,,\quad \mbox{ if } u\in Dom(\mathcal{L})\,,\;v\in
Dom(\mathcal{E})\,,
$$ 
(\cite{MR1303354}, \cite{MR1214375}). Therefore, equality holds 
in (\ref{eq:tf}) if $u=w\in Dom(\mathcal{L})$, see
\cite[(1.13.c)]{Fitz}, (\ref{eq:dpf}).
If $w\notin Dom(\mathcal{L})$, or $\mathcal{L}w$ does not belong to the
underlying $L^2$ space, then, as we shall see, the optimality of $\nu=-\mathcal{L}w/w$
critically depends on the choice of $w$.

According to \cite{BBC}, the Dirichlet form of the
censored $\alpha$-stable process in $D$ is
$$
\mathcal{C}(u,v)=\frac{1}{2}\mathcal{A}_{d,-\alpha}
\int_D \!\int_D
\frac{(u(x)-u(y))(v(x)-v(y))}{|x-y|^{d+\alpha}} \,dx\,dy\,,
$$ 
with core $C^\infty_c(D)$ (smooth functions in $C_D(D)$), 
the Lebesgue measure as the reference measure,
and the following {\it regional} fractional Laplacian on $D$ as the
generator (\cite[(3.12)]{BBC}, \cite{MR2238879,MR2214908}):
$$
\Delta^{\alpha/2}_D u(x)=\mathcal{A}_{d,-\alpha}\lim_{\varepsilon\to 0^+}
\int_{D\cap \{|y-x|>\varepsilon\}}\frac{u(y)-u(x)}{|x-y|^{d+\alpha}} \,dy\,.
$$ 
Here $\mathcal{A}_{d,-\alpha}=\Gamma((d+\alpha)/2)/
(2^{-\alpha}\pi^{d/2}|\Gamma(-\alpha/2)|)$, 
Clearly, (\ref{Hardyin}) is equivalent to 
\begin{equation}\label{Hardyinc}
\mathcal{C}(u,u)
\geq 
\mathcal{A}_{d,-\alpha}
{\kappa_{d,\alpha}}\,
\int_D {u^2(x)}
\;
x_d^{-\alpha}dx\,.
\end{equation}
Recall that the Dirichlet form of the stable process killed when leaving $D$ is
$$\mathcal{K}(u,v)=
\frac{1}{2}\mathcal{A}_{d,-\alpha}
\int_\Rd \!\int_\Rd
\frac{(u(x)-u(y))(v(x)-v(y))}{|x-y|^{d+\alpha}} \,dx\,dy\,,
$$
with core $C^\infty_c(D)$, 
the Lebesgue measure as the reference measure,
and the fractional Laplacian (on $\Rd$) as the generator,
$$
\Delta^{\alpha/2} u(x)=\mathcal{A}_{d,-\alpha}\lim_{\varepsilon\to 0^+}
\int_{\Rd\cap \{|y-x|>\varepsilon\}}\frac{u(y)-u(x)}{|x-y|^{d+\alpha}} \,dy
$$ 
(see, e.g., \cite{BBC}).
Decomposing $\Rd=D\cup D^c$, one obtains
$$
\mathcal{K}(u,u)=\mathcal{C}(u,u)+\int_Du^2(x)\kappa_D(x)dx\,,\quad
u\in C^\infty_c(D)\,,
$$
where (the density of the killing measure for $D$ is)
$$
\kappa_D(x)=\int_{D^c}\mathcal{A}_{d,-\alpha}
|x-y|^{-d-\alpha} \,dy=
\frac{1}{\alpha}\mathcal{A}_{d,-\alpha}
 \frac{\pi^\frac{d-1}{2}  \Gamma(\frac{1+\alpha}{2})}
  {\Gamma(\frac{\alpha+d}{2})}\,
\,x_d^{-\alpha}  \,, 
$$
see \cite[(2.3), (5.4)-(5.6)]{BBC}. It follows from (\ref{Hardyinc})
and Theorem \ref{hardy} that
\begin{eqnarray}
\nonumber
\mathcal{C}(u,u)
&\geq& 
\mathcal{A}_{d,-\alpha}
(\kappa_{d,\alpha}+
\frac{1}{\alpha}
 \frac{\pi^\frac{d-1}{2}  \Gamma(\frac{1+\alpha}{2})}
  {\Gamma(\frac{\alpha+d}{2})}
)
\int_D {u^2(x)}
x_d^{-\alpha}dx\\
&=&
\frac{\Gamma^2(\frac{1+\alpha}{2})}{\pi}
\int_D {u^2(x)}x_d^{-\alpha}dx\,,
  \label{eq:bckp}
\end{eqnarray}
for all $u\in C^\infty_c(D)$,
and the constant ${\Gamma^2(\frac{1+\alpha}{2})/\pi}$ 
is the best possible.

We like to note that in some respects, the censored stable process 
is a better analogue of the killed Brownian motion than the killed
stable process is (see \cite{ChenSong, BBC, MR2307059}, and
\cite{MR2114264, MR1991120}).
We suggest the former as a possible setup for studying Dirichlet boundary value problems for
non-local integro-differential operators and the corresponding stochastic processes
(\cite{MR1833696}, \cite{MR2365348}) on subdomains of $\Rd$
(\cite{BCI}), beyond the ``convolutional'' case of the whole of $\Rd$ (\cite{FLS,MR2373619}).
In this connection,  we refer to \cite{MR2238879,MR2214908,guan-2007} for
Green-type formulas for the censored process.

The reader interested in {fractional} Hardy inequalities may
consult \cite{KufnerPersson, HKP, MR1676324, ChenSong, Dijm2004,Dyda2}.
In particular, (\ref{Hardyin}) improves a part of
the (one-dimensional) result given in \cite[Theorem 2]{MR1676324}.
The fractional Hardy inequality on the whole of $\Rd$ is known as
Hardy-Rellich inequality, and the best
constant in this inequality was calculated in
\cite{MR0436854,MR1717839} (see also
\cite{MR2373619} for Pitt's inequality). 
As seen in \cite{Dijm2004}, the asymptotics of the
measure ${\dist(x, D^c)^{-\alpha}}dx$ agrees well with the homogeneity
of the kernel $|y-x|^{-d-\alpha}$ in (\ref{Hardyin}). Noteworthy, if $\alpha\leq 1$ and
$D$ is a {\it bounded} Lipschitz domain, 
then the best constant in (\ref{Hardyin}) is zero (\cite{Dijm2004}).

We like to make a few further remarks.
Theorem~\ref{hardy} and the
results obtained to date for Laplacian and fractional Laplacian suggest
possible strengthenings to weights with {additional} 
terms of lower-order boundary asymptotics (\cite{MR1655516,
  ghoussoub-2007, MR2373619}),
and
extensions to other 
specific or more general 
domains (\cite{MR1458330,MR2373619}).
To discuss the latter problem, we consider open $\Omega\subset D$, 
and its Hardy constant, $\kappa(\Omega)$, defined as the
largest number such that 
$$
\frac{1}{2} \int_\Omega \!\int_\Omega \frac{(u(x)-u(y))^2}{|x-y|^{d+\alpha}}
 \,dx\,dy 
\geq
  \kappa(\Omega) \int_\Omega \frac{u^2(x)}{\dist(x, \Omega^c)^\alpha}
  dx\,,\quad u\in C_c(\Omega)\,.
$$
Note that $\kappa(\Omega)>0$ if $\Omega$ is a bounded Lipschitz domain
and $\alpha>1$ (\cite{Dijm2004}).
Let $u\in C_c(\Omega)\subset C_c(D)$. We have
$$
\frac{1}{2} \int_\Omega \!\int_\Omega \frac{(u(x)-u(y))^2}{|x-y|^{d+\alpha}} \,dx\,dy 
 \leq
 \frac{1}{2}\int_D \!\int_D \frac{(u(x)-u(y))^2}{|x-y|^{d+\alpha}} \,dx\,dy\,,
$$
and
$$
\int_\Omega \frac{u^2(x)}{\dist(x, \Omega^c)^\alpha} dx
\geq
\int_D \frac{u^2(x)}{x_d^\alpha} dx\,, 
$$
thus $\kappa(\Omega)\leq \kappa_{d,\alpha}$.
We conjecture that
$\kappa(\Omega)=\kappa_{d,\alpha}$ for $\alpha\in (1,2)$ and {\it convex} $\Omega$, 
see \cite[Theorem 11]{MR1458330} for case of the Dirichlet of Laplacian.

Examining (\ref{eq:ns}) we see that $\kappa_{d,\alpha}\to \infty$ if
$\alpha\to 2$. This corresponds
to the fact that the only function $u\in C_c(\Omega)$
for which the left hand side of (\ref{Hardyin}) is finite for $\alpha=2$
is the zero function, see \cite{Brezis, Dijm2004}.
However, $\mathcal{A}_{d,-\alpha}\kappa_{d,\alpha}\to 1/4$ and
${\Gamma^2(\frac{1+\alpha}{2})/\pi}\to 1/4$ as
$\alpha\to 2$,  
an agreement with the classical Hardy inequality for 
Laplacian (\cite{MR1655516}) 
related to the fact that for $u\in
C^\infty_c(D)$, $\Delta_D^{\alpha/2}u\to \Delta u$ 
and $\mathcal{C}(u,u)\to -\int \Delta u(x) u(x)dx=\int
|\nabla u(x)|^2dx$ as $\alpha\to 2$ (the latter holds by Taylor's
expansion of order $2$, and a similar result is valid for $\mathcal{K}$). 
For the vast literature concerning optimal weights and constants in the
{classical} Hardy inequalities, and their applications we refer to \cite{MR0121449, MR1747888,
  MR1655516, FlorkiewiczKuchta, ghoussoub-2007, MR2079769, MR1920417}.

Our primary motivation to study Hardy
inequalities for non-local Dirichlet forms stems from the fact that 
the converse of (\ref{eq:tf}) stated in \cite[Theorem 1.9]{Fitz} 
allows for a construction of superharmonic functions, or barriers (\cite{Ancona}), when a Hardy
inequality is given. These functions may then be used to investigate transience
and boundary behavior of the underlying Markov processes
(\cite{Ancona}, \cite{BBC}, \cite{Dyda2}, \cite{ChenSong}).
In particular, we expect that the results of 
\cite{Dijm2004, Dyda2} may be used to obtain, for
the anisotropic stable (\cite{MR2320691,MR2137058})
censored processes, the
ruin probabilities generalizing \cite[Theorem 5.10]{BBC}, and to develop 
the boundary potential theory on Lipschitz domains (\cite{BBC,
  MR2214908, guan-2007}) in analogy
with those of the killed stable processes (\cite{MR2365478, MR1654115, MR1438304}). 
We also like to mention the connection of optimal Hardy inequalities
with critical Schr\"odinger perturbations and the so-called ground state representation \cite{FLS}.

Despite the general context mentioned above, the paper is essentially 
self-contained and purely analytic. In particular we directly derive
Fitzsimmons' ratio measure by a simple manipulation with quadratic
expressions, (\ref{eq:dpf}), not unrelated to the ground state
representation of \cite[(4.2)]{FLS}.
Theorem~\ref{hardy} is proved below in this section.
In Section~\ref{sec:a} we calculate auxiliary integrals.

In what follows, $|x|=(x_1^2+\dots+x_d^2)^{1/2}$ denotes the Euclidean norm of 
$x=(x_1,\ldots,x_d)\in \R^d$, and $B(x,r)$ denotes the Euclidean ball of radius $r>0$ centered at $x$.
For $d\geq 2$ we occasionally write $x=(x',x_d)$, where
$x'=(x_1,\ldots,x_{d-1})$, and we let $\|x'\|=\max_{k=1,\ldots,d-1}
|x_k|$, the supremum norm on $\R^{d-1}$.
\begin{proof}[Proof of Theorem~\ref{hardy}]
For $u, v\in C^\infty_c(D)$ we define (Dirichlet form) 
$$
\mathcal{E}(u,v)=\frac{1}{2}
\int_D \!\int_D
\frac{(u(x)-u(y))(v(x)-v(y))}{|x-y|^{d+\alpha}} \,dx\,dy\,,
$$ 
and (its generator)
$$
\mathcal{L} u(x)=\lim_{\varepsilon\to 0^+}
\int_{D\cap \{|y-x|>\varepsilon\}}\frac{u(y)-u(x)}{|x-y|^{d+\alpha}} \,dy\,,
$$ 
so that $\mathcal{A}_{d,-\alpha}\mathcal{E}=\mathcal{C}$,
$\mathcal{A}_{d,-\alpha}\mathcal{L}=\Delta^{\alpha/2}_D$, and
$\mathcal{E}(u,u)$ equals the left-hand side of (\ref{Hardyin}).

Let $p\in(-1,\alpha)$, $x=(x_1,\ldots,x_d)\in D$,
$$
w_p(x)=
  x_d^p\,. 
$$
By \cite[(5.4) and (5.5)]{BBC},
\begin{equation}\label{ullapl}
 \mathcal{L}w_p(x) = 
\;\gamma(\alpha,p)\, 
 \frac{\pi^\frac{d-1}{2}  \Gamma(\frac{1+\alpha}{2})}
  {\Gamma(\frac{\alpha+d}{2})}\, x_d^{-\alpha}\,w_p(x) \,,
\end{equation}
where the (absolutely convergent) integral 
\begin{equation}
  \label{eq:dg}
 \gamma(\alpha,p) = \int_0^1 
  \frac{(t^p-1)(1-t^{\alpha-p-1})}{(1-t)^{1+\alpha}}\, dt  \,,
\end{equation}
is negative if $p(\alpha-p-1)>0$.
Guided by the discussion in Section~\ref{sec:i} we let 
\begin{equation}
  \label{eq:wnn}
 \nu(x) = 
\frac{-\mathcal{L}w_p(x)}{w_p(x)} = 
   -\gamma(\alpha,p)
 \frac{\pi^\frac{d-1}{2}  \Gamma(\frac{1+\alpha}{2})}
  {\Gamma(\frac{\alpha+d}{2})}\,
\,x_d^{-\alpha}  \,. 
\end{equation}
Since, for each $t\in(0,1)$, the function 
$$
 p \mapsto  \frac{(t^p-1)(1-t^{\alpha-p-1})}{(1-t)^{1+\alpha}}
$$
is convex and symmetric with respect to $(\alpha-1)/2$,
therefore $p\mapsto \gamma(\alpha,p)$ has a non-positive minimum at
$p=(\alpha-1)/2$. 
By Lemma~\ref{propapp} below, (\ref{eq:wnn}), and
(\ref{eq:tf}), we obtain (\ref{Hardyin}) for $u\in
C^\infty_c(D)\subset Dom(\mathcal{C})$, with $\kappa_{d,\alpha}$ given by (\ref{eq:ns}).
The case of general $u\in C_c(D)$ is obtained by an approximation.

Since the setups of \cite{Fitz} and \cite{BBC} are rather
complex, we like to give the following elementary proof of (\ref{Hardyin}).
Let $w=w_{(\alpha-1)/2}$, $u\in C_c(D)$, $x,y\in D$. We
have
\begin{eqnarray} 
\left(u(x)-u(y)\right)^2&+&u^2(x)\frac{w(y)-w(x)}{w(x)}  
+u^2(y)\frac{w(x)-w(y)}{w(y)}\nonumber \\
&=&
{w(x)w(y)}
{\left[u(x)/w(x)-u(y)/w(y)\right]^2}
\geq 0\,.  \label{eq:dpf}
\end{eqnarray}
We integrate (\ref{eq:dpf}) against the symmetric measure
$1_{|y-x|>\varepsilon}|x-y|^{-d-\alpha}\,dx\,dy$, and we let
$\varepsilon \to 0^+$. According to the calculations above,
\begin{eqnarray*}
\frac{1}{2}\int\limits_D\int\limits_D \frac{(u(x)-u(y))^2} {|x-y|^{d+\alpha}}\,dx\,dy
&\geq& \int\limits_D u^2(x)
\lim\limits_{\varepsilon \to 0^+}
\!\!\!\!\!\!
\int\limits_{\{y\in D:\, |y-x|>\varepsilon\}}
\!\!\!\!\!\!\!\!\!
\frac{w(x)-w(y)}{|y-x|^{d+\alpha}}\,dy\,
\frac{dx}{w(x)} \\
&=& \,\kappa_{d,\alpha}\int\limits_D u^2(x)
\,x_d^{-\alpha}\,dx\,.
\end{eqnarray*}

To complete the proof we will verify the
optimality of $\kappa_{d,\alpha}$.
In what follows we denote ${\bf p}=\frac{\alpha-1}{2}$.
If $\alpha\geq 1$ then we consider functions $v_n$ such that
\begin{itemize}
\item[(i)] $v_n=1$ on $[-n^2,n^2]^{d-1}\times [\frac{1}{n}, 1]$,
\item[(ii)] $\supp v_n \subset [-n^2-1,n^2+1]^{d-1}\times
  [\frac{1}{2n}, 2]$,
\item[(iii)] $0\leq v_n \leq 1$,
 $|\nabla v_n(x)|\leq cx_d^{-1}$ and
 $|\nabla^2 v_n(x)|\leq cx_d^{-2}$ for $x\in D$.
\end{itemize}
If $\alpha< 1$ then we stipulate 
\begin{itemize}
\item[(i')] $v_n=1$ on $[-n^2,n^2]^{d-1}\times [1, n]$,
\item[(ii')] $\supp v_n \subset [-n^2-n,n^2+n]^{d-1}\times
  [\frac{1}{2}, 2n]$,
\item[(iii)] $0\leq v_n \leq 1$,
 $|\nabla v_n(x)|\leq cx_d^{-1}$ and
 $|\nabla^2 v_n(x)|\leq cx_d^{-2}$ for $x\in D$.
\end{itemize}
We define (for any $\alpha\in (0,2)$),
\begin{equation}
  \label{eq:duu}
u_n(x)=v_n(x)x_d^{\bf p}\,.  
\end{equation}
We have
$$
\int_D \frac{u_n(x)^2}{x_d^\alpha}\,dx \geq
\int\limits _{\{x:\,\|x'\|\leq n^2,\,\frac{1}{n}<x_d<1\}}
 \frac{x_d^{2{\bf p}}}{x_d^\alpha}\,dx = (2n^2)^{d-1}\log n.
$$
Thus, by Lemma~\ref{meczacy2} below, $\kappa_{d,\alpha}$ may not be replaced in (\ref{Hardyin})
by a bigger constant. 
\end{proof}

\section{Appendix}\label{sec:a}
\begin{lem}\label{propapp} 
For $0<\alpha<2$,
\begin{equation}
  \label{eq:www}
\gamma(\alpha,\frac{\alpha-1}{2}) = 
-\frac{1}{\alpha}\left[
B(\frac{1+\alpha}{2},\frac{2-\alpha}{2})2^{-\alpha}-1\right]\,.
\end{equation}
\end{lem}
\begin{proof}
Since $$
 \gamma(\alpha,p) = \int_0^1 
  \frac{t^p-t^{\alpha-1}-1+t^{\alpha-p-1}}{(1-t)^{1+\alpha}}\, dt\,,
$$
we are led to considering
$$
B_\kappa(a,b)=\int_0^\kappa t^{a-1}(1-t)^{b-1}\,dt\,.
$$
Here and below $a>0$, $b>-2$, and $0\leq \kappa<1$. We will also
assume that $b\neq 0, 1$.

Using $t^{a-1}=t^{a-1}(1-t)+t^a$, and integration by
parts, we get
$$
B_\kappa(a,b)=\frac{a+b}{b}B_\kappa(a,b+1)-\frac{1}{b}\kappa^a(1-\kappa)^b\,,
$$
hence
$$
B_\kappa(a,b)=\frac{a+b}{b}
\left(\frac{a+b+1}{b+1}B_\kappa(a,b+2)-\frac{1}{b+1}\kappa^a(1-\kappa)^{b+1}\right)
-\frac{1}{b}\kappa^a(1-\kappa)^b\,.
$$
Clearly, $\gamma(\alpha,p)=\lim_{\kappa\to 1^-}\left[
B_\kappa(p+1,-\alpha)-B_\kappa(\alpha,-\alpha)-B_\kappa(1,-\alpha)+B_\kappa(\alpha-p,-\alpha)
\right].$
For $\alpha\neq 1$ we have, 
\begin{eqnarray*}
&&
B_\kappa(p+1,-\alpha)-B_\kappa(\alpha,-\alpha)-B_\kappa(1,-\alpha)+B_\kappa(\alpha-p,-\alpha)=
\frac{1}{\alpha(\alpha-1)}\times\\
&&
\left\{
(p+1-\alpha)(p+1-\alpha+1)B_\kappa(p+1,2-\alpha)
-(\alpha-\alpha)(\alpha-\alpha+1)B_\kappa(\alpha,2-\alpha)\right.\\
&&\left.-(1-\alpha)(1-\alpha+1)B_\kappa(1,2-\alpha)
+(\alpha-p-\alpha)(\alpha-p-\alpha+1)B_\kappa(\alpha-p,2-\alpha)\right\}\\
&&
+\frac{(1-\kappa)^{1-\alpha}}{\alpha(\alpha-1)}
\left[
-(p+1-\alpha)\kappa^{p+1}
+(\alpha-\alpha)\kappa^\alpha
+(1-\alpha)\kappa^1
-(\alpha-p-\alpha)\kappa^{\alpha-p}\right]\\
&&
+\frac{(1-\kappa)^{-\alpha}}{-\alpha}
\left[
-\kappa^{p+1}
+\kappa^\alpha
+\kappa^1
-\kappa^{\alpha-p}\right]\,.
\end{eqnarray*}
All expressions in the
square brackets, and their derivative, vanish at $\kappa=1$.
Thus, they do not contribute to the limit as $\kappa\to 1$.
For $\alpha\neq 1$ we get 
\begin{eqnarray}
\gamma(\alpha,p)&=&\frac{1}{\alpha(\alpha-1)}\left\{
(p+1-\alpha)(p+2-\alpha)B(p+1,2-\alpha)\right.\nonumber\\
&&\left.-(1-\alpha)(2-\alpha)B(1,2-\alpha)
+p(p-1)B(\alpha-p,2-\alpha)\right\}\,.\label{eq:wgap}
\end{eqnarray}
By the duplication formula
$\Gamma(2z)=(2\pi)^{-1/2}\, 2^{2z-1/2}\, \Gamma(z)\, \Gamma(z+1/2)$
with $2z=2-\alpha$, for $p=(\alpha-1)/2$, this equals
\begin{eqnarray*}
&&\frac{1}{\alpha}\left[-\frac{3-\alpha}{2}B(\frac{\alpha+1}{2},2-\alpha)+1\right]=
\frac{1}{\alpha}\left[-\Gamma(\frac{\alpha+1}{2})\Gamma(2-\alpha)/\Gamma(\frac{3-\alpha}{2})+1\right]\\
&&=
\frac{1}{\alpha}\left[-\Gamma(\frac{\alpha+1}{2})\Gamma(\frac{2-\alpha}{2})/\Gamma(\frac{1}{2})2^{1-\alpha}+1\right]
=
-\frac{1}{\alpha}\left[B(\frac{\alpha+1}{2},\frac{2-\alpha}{2})2^{-\alpha}-1\right]\,.
\end{eqnarray*}
We thus proved (\ref{eq:www}) for $\alpha\neq 1$. 
The case of $\alpha=1$ is trivial. In fact, $\gamma(1, 0)=0$.
\end{proof}

\begin{lem}\label{techniczny}
Let $-1<r<\alpha<2$ and $\alpha>0$.
There exists a constant $c$
such that
$$
 \int_{D\setminus B(x,a)} \frac{y_d^r}{|x-y|^{d+\alpha}}\,dy
 \leq ca^{-\alpha} (a \vee x_d)^r
$$
for every $a>0$ and $x\in D$.
\end{lem}
\begin{proof}
Let $B(x,s,t)=B(x,t)\setminus B(x,s)$.
If $a\geq x_d/2$ then
\begin{eqnarray*}
\int_{D\setminus B(x,a)}
 \frac{y_d^r}{|x-y|^{d+\alpha}}\,dy &\leq&
 c \sum_{k=0}^\infty \int_{D\cap B(x,2^ka, 2^{k+1}a)}
 \frac{y_d^r}{(2^k a)^{d+\alpha}}\,dy \\&\leq&
c' \sum_{k=0}^\infty (2^ka)^{r-\alpha} = c''a^{r-\alpha}.
\end{eqnarray*}
If $a<x_d/2$ then
$$
\int_{D\cap B(x,a,x_d)}
 \frac{y_d^r}{|x-y|^{d+\alpha}}\,dy \leq
 cx_d^r\, a^{-\alpha}\,,
$$
and, by first part of the proof,
$$
\int_{D\setminus B(x,x_d)}
 \frac{y_d^r}{|x-y|^{d+\alpha}}\,dy \leq
c x_d^{r-\alpha}.
$$
This ends the proof.
\end{proof}
Recall that ${\bf p}=\frac{\alpha-1}{2}$, and $u_n$ is defined by 
(\ref{eq:duu}).
\begin{lem}\label{meczacy2}
There exists a constant $c$ 
independent of $n$,
such that
$$
 \int_D\!\int_D  \frac{(u_n(x)-u_n(y))^2}{|x-y|^{d+\alpha}}\,dy\,dx \leq
 cn^{2(d-1)} + 2\kappa_{d,\alpha} \int\limits_D u_n^2(x) \,x_d^{-\alpha}\,dx.
$$
\end{lem}
\begin{proof}
To simplify the notation we let $K_n = \supp u_n$ and $u=u_n$, $v=v_n$.
By (\ref{eq:dpf}) and (\ref{ullapl}) we have
\begin{eqnarray*}
\int\limits_D\int\limits_D \frac{(u(x)-u(y))^2} {|x-y|^{d+\alpha}}\,dx\,dy
&=& 2\kappa_{d,\alpha} \int\limits_D u^2(x) \,x_d^{-\alpha}\,dx \\
&& + \int\limits_D\int\limits_D \frac{(v(x)-v(y))^2} {|x-y|^{d+\alpha}}\,w(x)w(y)\,dx\,dy.
\end{eqnarray*}
We will estimate the latter (double) integral by $cn^{2(d-1)}$, by
splitting it into the sum of the following six integrals $I_1+\ldots+I_6$.

We will first consider the case of $\alpha\geq 1$.

If $x\in K_n$ and $y\in B(x,\frac{1}{4n})$, 
then $|v(x)-v(y)|\leq c|x-y| x_d^{-1}$, as follows from (ii)
and (iii). We thus have
\begin{eqnarray*}
I_1&=& \int_D \int_{B(x,\frac{1}{4n})}
\frac{(v(x)-v(y))^2}{|x-y|^{d+\alpha}}\,w(x)w(y)\,dy\,dx\\
 &\leq&
2\int_{K_n} \int_{B(x,\frac{1}{4n})}
\frac{(v(x)-v(y))^2}{|x-y|^{d+\alpha}}\,w(x)w(y)\,dy\,dx\\
 &\leq&
c \int_{K_n} \int_{B(x,\frac{1}{4n})} 
  \frac{x_d^{2{\bf p}-2}}{|x-y|^{d+\alpha-2}} \,dy\,dx \\
&\leq& c' n^{2(d-1)}.
\end{eqnarray*}
A similar argument gives
\[
I_2=
\int_{\{x:\,x_d\geq \frac{1}{2}\}} \int_{B(x,\frac{1}{4})}
\frac{(v(x)-v(y))^2}{|x-y|^{d+\alpha}}\,w(x)w(y)\,dy\,dx \\
\leq
cn^{2(d-1)}.
\]
We then have by Lemma~\ref{techniczny} for $a=1/4$ and $r={\bf p}$
\begin{eqnarray*}
I_3&=&
\int_D \int_{D\setminus B(x,\frac{1}{4})}
\frac{(v(x)-v(y))^2}{|x-y|^{d+\alpha}}\,dy\,dx\\
 &\leq&
\int_{K_n} \int_{D\setminus B(x,\frac{1}{4})}
\frac{c}{|x-y|^{d+\alpha}}\,w(x)w(y)\,dy\,dx \\ &\leq&
c'n^{2(d-1)}.
\end{eqnarray*}
If $d\geq 2$ then we consider $P_n = \{x\in\R^d : \,\|x'\|\geq n^2-1\,,\;
0<x_d<\frac{1}{2} \}$ and
$P_n^0 = P_n \cap \{x\in\R^d : \,\|x'\| < n^2+\frac{5}{4} \}$.
We obtain
\begin{eqnarray*}
I_4&=&
\int_{P_n} \int_{D\cap B(x,\frac{1}{4n},\frac{1}{4})}
\frac{(v(x)-v(y))^2}{|x-y|^{d+\alpha}}\,w(x)w(y)\,dy\,dx \\
&\leq&
 \int_{P_n^0}  \int_{D\setminus B(x,\frac{1}{4n})}
\frac{c}{|x-y|^{d+\alpha}}\,dy\,dx \\
&\leq& c' |P_n^0| n^\alpha
\leq c''n^{2(d-1)}.
\end{eqnarray*}
We let $R_n = \{x\in\R^d : \,\|x'\| < n^2-1\,,\;0<x_d<\frac{2}{n} \}$ 
if $d\geq 2$, and we let
$R_n = \{x\in\R : \,0<x<\frac{2}{n} \}$ if $d=1$. We have
\begin{eqnarray*}
I_5&=&\int_{R_n} \int_{D \cap B(x,\frac{1}{4n},\frac{1}{4})}
\frac{(v(x)-v(y))^2}{|x-y|^{d+\alpha}}\,w(x)w(y)\,dy\,dx \\
&\leq&
 \int_{R_n} \int_{D\setminus B(x,\frac{1}{4n})}
\frac{cy_d^{{\bf p}} (\frac{1}{n})^{\bf p}}{|x-y|^{d+\alpha}}\,dy\,dx 
\leq c'n^{2(d-1)}.
\end{eqnarray*}
In the last inequality above we have used Lemma~\ref{techniczny}
with $a=\frac{1}{4n}$ and $r={\bf p}$.

We define $L_n=\{x\in \R : \frac{2}{n} \leq x < \frac{1}{2} \}$ in
dimension $d=1$, and  for $d\geq 2$ we let
$L_n=\{x\in \R^d : \|x'\| < n^2-1\,,\;\frac{2}{n} \leq x_d <\frac{1}{2} \}$.
We have
\begin{eqnarray*}
I_6 &=&
\int_{L_n} \int_{D\cap B(x,\frac{1}{4n},\frac{1}{4})}
\frac{(v(x)-v(y))^2}{|x-y|^{d+\alpha}}\,w(x)w(y)\,dy\,dx\\
&\leq&
\int_{L_n} \int_{\{y:0<y_d<\frac{1}{n}\}}
\frac{w(x)w(y)}{|x-y|^{d+\alpha}}\,dy\,dx
\end{eqnarray*}
For $d\geq 2$ and $x\in L_n$ we have
\begin{eqnarray*}
&& \int\limits_{\{y: 0<y_d<\frac{1}{n} \}}
  \frac{dy}{|x-y|^{d+\alpha}} 
 \leq
c  \int\limits_{\{y: 0<y_d<\frac{1}{n} \}}
  \frac{dy}{(|x'-y'|^2 + x_d^2)^{(d+\alpha)/2}} \\
&=&
\frac{c}{n} 
  \left( 
 \int_{\{y'\in\R^{d-1}: |x'-y'|<x_d\}} +
 \int_{\{y'\in\R^{d-1}: |x'-y'|\geq x_d\}}
  \right)
  \frac{dy'}{(|x'-y'|^2 + x_d^2)^{(d+\alpha)/2}} \\
&\leq& c' \frac{x_d^{-\alpha-1}}{n}\,,
\end{eqnarray*}
thus
$$
I_6 \leq
 c \int_{L_n} 
\left(\frac{x_d}{n}\right)^{\bf p}
\frac{x_d^{-\alpha-1}}{n} \,dx
\leq c'n^{2(d-1)}.
$$
The case of $d=1$ is left to the reader. 

We now consider the case of $\alpha< 1$.
We have
\begin{eqnarray*}
I&=&\int\limits_D\int\limits_D \frac{(v(x)-v(y))^2} {|x-y|^{d+\alpha}}\,w(x)w(y)\,dx\,dy\\
&\leq&
\int_D \int_{B(x,\frac{1}{4})}
+ \int_{\{x:\,x_d\geq \frac{n}{2}\}} \int_{B(x,\frac{n}{4})}
+ \int_D \int_{D\setminus B(x,\frac{n}{4})}
+ \int_{P_n} \int_{D\cap B(x,\frac{1}{4},\frac{n}{4})}\\
&&+ \int_{\{x:0<x_d<2\}} \int_{D \cap B(x,\frac{1}{4},\frac{n}{4})}
+ \int_{L_n} \int_{D\cap B(x,\frac{1}{4},\frac{n}{4})}\\
&=&I_1+I_2+I_3+I_4+I_5+I_6,
\end{eqnarray*}
where
\begin{eqnarray*}
P_n &=& \{x\in\R^d : \,\|x'\|\geq n^2-n\,,\; 0<x_d<\frac{n}{2} \};\\
L_n&=&\{x\in \R^d : \|x'\| < n^2-n\,,\; 2 \leq x_d <\frac{n}{2} \},
\end{eqnarray*}
for $d\geq 2$, and $P_n=\emptyset$, $L_n=(2,\frac{n}{2})$ for $d=1$.
We estimate the integrals $I_k$ in a similar way as for $\alpha\geq 1$.
The details are left to the reader. 
\end{proof}
Similar but simpler estimates were
given in \cite{Dijm2004} to prove that 
the Hardy constant of a bounded Lipschitz domain (e.g. of an interval) is
zero if $\alpha\leq 1$. 
We also like to mention that there is an alternative proof of
Lemma~\ref{meczacy2} (not given here), which explicitly
uses the fact that that $w^2(x)=x_d^{\alpha-1}$ is harmonic (\cite{BBC}) for
$\Delta^{\alpha/2}_D$. Similarly, the best constant, $1/4$, in the
classical Hardy inequality for the half-space $D$ is obtained by
considering $w(x)=\sqrt{x_d}$ in Fitzsimmons' ratio $\nu=-\Delta w/w$.

\begin{ack}
{\rm
We thank Grzegorz Karch for a discussion on recent results on
non-local operators, and Rodrigo Ba\~nuelos, Ma{\l}gorzata Kuchta and Jacek Zienkiewicz for 
discussions on classical Hardy inequality. 
}  
\end{ack}

\bibliographystyle{abbrv}
\bibliography{bibca}

\end{document}